\documentclass[utf8,reqno]{amsart}
\RequirePackage[utf8]{inputenc}
\usepackage{amsmath,amsthm,amsfonts,amssymb,amscd,amsbsy,multirow,hyperref}
\usepackage[all]{xy}
\usepackage{graphicx,pgf,caption,subcaption}
\usepackage{enumerate}
\usepackage{dsfont}
\usepackage{stmaryrd}
\usepackage{ytableau}
\usepackage{fancyhdr}
\usepackage{CJK}
\usepackage{hyperref}

\allowdisplaybreaks

\hypersetup{
    pdfnewwindow=true,
    colorlinks=true,
    linkcolor=blue,
    citecolor=blue,
    urlcolor=black,
}

\newtheorem{theorem}{Theorem}[]

\newtheorem{proposition}[theorem]{Proposition}
\newtheorem{corollary}[theorem]{Corollary}

\theoremstyle{definition}
\newtheorem{definition}[theorem]{Definition}

\theoremstyle{remark}

\allowdisplaybreaks
\numberwithin{equation}{section}
\numberwithin{theorem}{section}

\begin{document}
\begin{CJK*}{GBK}{song}

\title{The extended quasi-Einstein manifolds with generalised Ricci solitons}

\author[Z. M. Huang]{Zhiming Huang}
\address{Zhiming Huang \\ School of Mathematics and Physics \\ Guangxi Minzu University \\ Guangxi Nanning 530006, China}
\email{hzm1584483249@163.com}

\author[W. J. Lu]{Weijun Lu}
\address{Weijun Lu \\ School of Mathematics and Physics \\ Guangxi Minzu University \\ Guangxi Nanning 530006, China}
\email{weijunlu2008@126.com}

\author[F. H. Su]{Fuhong Su}
\address{Fuhong Su \\ School of Mathematics and Physics \\ Guangxi Minzu University \\ Guangxi Nanning 530006, China}
\email{fuhongsu0514@163.com}

\maketitle
\vspace{-0.3cm}

\begin{abstract}
As a generalization of Einstein manifolds, the nearly quasi-Einstein manifolds and pseudo quasi-Einstein manifolds are both interesting and useful in studying the general relativity. In this paper, we study the extended quasi-Einstein manifolds which derive from pseudo quasi-Einstein manifolds. After showing the existence theorem of extended quasi-Einstein manifold, we give some special geometric properties of such manifolds. At the same time, we also discuss the extended quasi-Einstein manifolds with certain soliton like generalised Ricci soliton or Riemann soliton. Furthermore, we construct some nontrivial example to illustrate these extended quasi-Einstein manifolds.
\end{abstract}

\footnotetext{keywords: nearly quasi-Einstein manifolds, pseudo quasi-Einstein manifolds, extended quasi-Einstein manifolds, generalised Ricci soliton, Riemann soliton}

\section{Introduction}
A non-flat Riemannian manifold $({{M}^{n}},g)$ $(n>2)$ is said to be an Einstein manifold if the following condition
\begin{equation}\label{eq:a1} \text{Ric}(X,Y)=\frac{r}{n}\,g(X,Y), \quad \forall X, Y \in {{C}^{\infty }}(TM), \end{equation}
hold on $M$, where \emph{r} and \text{Ric} denote the scalar curvature and the Ricci tensor of $({{M}^{n}},g)$, respectively. According to \cite{article.8}, \eqref{eq:a1} is called the Einstein metric condition.

In $2008$, Chaki and De \textsuperscript{\cite{article.1}} introduced the notion of nearly quasi-Einstein manifolds which is a generalization of Einstein manifold. A non-flat Riemannian manifold $({{M}^{n}},g)$ $(n\ge 3)$ is said to be a nearly quasi-Einstein manifold if its Ricci tensor is not identically zero and satisfies the condition

\begin{equation}\label{eq:a2} \text{Ric}(X,Y)=ag(X,Y)+bE(X,Y), \quad \forall X, Y \in {{C}^{\infty }}(TM), \end{equation}
where $a$ and $b$ are non-zero scalars and $E$ is a non-zero $(0,2)$ tensor. We shall call $E$ the associated tensor and $a$ and $b$ as associated scalars. An $n$ dimensional nearly quasi Einstein manifold will be denoted by $E^n_{\rm{NQ}}$.

In \cite{article.9,article.20}, Singh and Pandey consider a type of nearly quasi Einstein manifold, whose associated tensor $E$ of type $(0,2)$ is given by
\[E(X,Y)=A(X)B(Y)+B(X)A(Y), \quad \forall X, Y \in {{C}^{\infty }}(TM),\]
where $A$ and $B$ are non-zero $1$-forms associated with orthogonal unit vector fields $V$ and $U$, i.e.,
\begin{equation}\label{eq:a7} g(X,U)=A(X), g(X,V)=B(X), g(V,U)=0, \quad \forall X \in {{C}^{\infty }}(TM).\end{equation}
Thus, the equation \eqref{eq:a2}, assumes the form
\begin{equation}\label{eq:a8} \text{Ric}(X,Y)=ag(X,Y)+b[A(X)B(Y)+B(X)A(Y)], \quad \forall X, Y \in {{C}^{\infty }}(TM).\end{equation}

De and Ghosh gave a similarly definition about this manifolds. They introduced a type of non-flat Riemannian manifold called mixed quasi-Einstein manifold \textsuperscript{\cite{article.3,article.4}}. A non-flat Riemannian manifold $({{M}^{n}},g)$ $(n>2)$ is called mixed quasi-Einstein manifold if its Ricci tensor of type $(0,2)$ is not identically zero and satisfies the equations \eqref{eq:a7}, \eqref{eq:a8} and $b\ne 0$. The vector fields $U$ and $V$ are called the generators of the manifold. Such a manifolds is denoted by $E^n_{\rm{MQ}}$. More works have been done in the mixed quasi-Einstein manifolds \textsuperscript{\cite{article.12,article.17}}.

A non-flat Riemannian manifold $({{M}^{n}},g)$ $(n\ge 3)$ is called a pseudo quasi-Einstein manifold if its Ricci tensor \text{Ric} of type $(0,2)$ is not identically zero and satisfies the following \textsuperscript{\cite{article.10,article.11}}:
\begin{equation}\label{eq:a9} \text{Ric}(X,Y)=ag(X,Y)+bA(X)A(Y)+cE(X,Y), \quad \forall X, Y \in {{C}^{\infty }}(TM), \end{equation}
where $a$, $b$ and $c$ are scalars of which $c$ is non-zero, and $A$ is a non-zero $1$-form such that $g(X,U) = A(X)$ for all vector fields $X$ with $U$ being a unit vector field called the generator of the manifold, $E$ is a symmetric $(0,2)$ tensor with vanishing trace and satisfying $E(X,U)=0$ for all vector fields $X$. Also $a$, $b$ and $c$ are called the associated scalars; $A$ is the associated $1$-form of the manifold and $E$ is called the associated tensor of the manifold. Such an $n$-dimensional manifold will be denoted by $E^n_{\rm{PQ}}$.

\begin{theorem}
\cite{article.28}(F.Schur) $(M^{n},g)$ $(n>2)$ is a connecting Riemannian manifold. If the sectional curvature of $(M^{n},g)$ at $p$ with respect to the plane $X\wedge Y\in T_{p}M$, dependent on the choices of point $p$ not basis $X\wedge Y$, i.e., $sec_{p}(X\wedge Y)=c(p)$, then $M$ is a manifold of constant curvature.
\end{theorem}
\begin{proof}
Let us suppose $R_{1}(W,Z,X,Y)=g(W,X)g(Z,Y)-g(Z,X)g(W,Y)$, we have
\[R(W,Z,X,Y)=c(p)R_{1}(W,Z,X,Y).\]
Hence
\begin{align*}
&R(W,Z,X,\nabla _{U}Y)+g(\nabla _{U}R(W,Z)X,Y)\\
=&\nabla _{U}R(W,Z,X,Y))=\nabla _{U}cR_{1}(W,Z,X,Y))\\
=&\nabla _{U}c[g(W,X)g(Z,Y)-g(Z,X)g(W,Y)]\\
=&(Uc)g(Y,g(W,X)Z-g(Z,X)W)+c[R_{1}(\nabla _{U}W,Z,X,Y)+R_{1}(W,\nabla _{U}Z,X,Y)\\
&+R_{1}(W,Z,\nabla _{U}X,Y)+R_{1}(W,Z,X,\nabla _{U}Y)]\\
=&(Uc)g(Y,g(W,X)Z-g(Z,X)W)+g(R(\nabla _{U}W,Z)X,Y)+g(R(W,\nabla _{U}Z)X,Y)\\
&+g(R(W,Z)\nabla _{U}X,Y)+R(W,Z,X,\nabla _{U}Y).
\end{align*}
We get
\begin{equation}\label{16}\begin{aligned}
&\nabla _{U}R(W,Z)X-R(\nabla _{U}W,Z)X-R(W,\nabla _{U}Z)X-R(W,Z)\nabla _{U}X\\
&=(Uc)(g(W,X)Z-g(Z,X)W).
\end{aligned}\end{equation}
From the \eqref{16}, we obtain
\begin{equation}\label{1}\begin{aligned}
&(Uc)(g(W,X)Z-g(Z,X)W)+(Wc)(g(Z,X)U-g(X,U)Z)\\
&+(Zc)(g(X,U)W-g(X,W)U)\\
&=(\nabla _{U}R)(W,Z)X+(\nabla _{W}R)(Z,U)X+(\nabla _{Z}R)(U,W)X=0.
\end{aligned}\end{equation}
Suppose that $W$, $X$, $Z$ are mutually orthogonal $C^\infty$ vector fields of which $Z$ is a unit vector field. Putting $U=X$ in \eqref{1}, we have
\[(Zc)W-(Wc)Z=0, \quad \forall Z, W \in {{C}^{\infty }}(TM).\]
Then $Zc=0$ and $Wc=0$.

Let $\left\{x^{i}\right\}$ be local coordinate systems with respect to orthonormal basis $\left\{X_{i}\right\}$, then $X_{i}c=0$. Hence
\[\frac{\partial }{\partial x^{i}}c=(\sum\limits_{j=1}^{n}{a_{i}^{j}{{X}_{j}})}c=\sum\limits_{j=1}^{n}{a_{i}^{j}({{X}_{j}}}c)=0, \quad {X}_{j} \in {{C}^{\infty }}(TM), \]
then $c$ is a local constant. Since $M$ is connecting, $M$ is a manifold of constant curvature.
\end{proof}
\begin{corollary}
A three dimensional Einstein manifold $(M^{3},g)$ is a manifold of constant sectional curvature.
\end{corollary}
From F.Schur theorem and three dimensional Einstein manifold has constant sectional curvature, motivated these facts and the ideas of quasi-Einstein manifolds with one generator, we extend to consider three generators into quasi-Einstein manifolds and induce a new notion so called extended quasi-Einstein manifold, from which we expect to uncover more information about sectional curvature and Ricci curvature and so on (etc).

The present paper is organized as follows. In section $2$, we prove the existence theorem of extended quasi-Einstein manifold and some basic geometric properties of such manifold are obtained. We consider the associated vector fields as Killing vector fields, parallel vector fields, concurrent vector fields, respectively. After that we discuss about the nature of the Ricci tensor satisfies the Codazzi type and Ricci semi-symmetric manifold. In section $3$, we study generalised Ricci soliton and Riemann soliton on extended quasi-Einstein manifold and obtain some characterizations. Finally, we provide an example of extended quasi-Einstein manifolds.

\section{Extended quasi-Einstein manifold and basic geometric properties}
In this section, we prove the existence theorem of extended quasi-Einstein manifold and some basic geometric properties of such manifold are obtained.

\begin{definition}
A non-flat Riemannian manifold $({{M}^{n}},g)$ $(n\ge 3)$ is called an extended quasi-Einstein manifold if its Ricci tensor of type $(0,2)$ is not identically zero and satisfies the condition
\begin{equation}\label{eq:a4} \text{Ric}(X,Y)=ag(X,Y)+bA(X)A(Y)+c(B(X)D(Y)+D(X)B(Y)), \end{equation}
for all $X$, $Y$ $\in {{C}^{\infty }}(TM)$. Here $a$, $b$, $c$ are scalars of which $c\ne 0$, and $A$, $B$, $D$ are non-zero $1$-forms such that
\begin{equation}\label{eq:a5} g(X,U)=A(X), g(X,V)=B(X), g(X,T)=D(X), \end{equation}
where $U$, $V$, $T$ are mutually orthogonal unit $C^\infty$ vector fields by using $\rm{dim}(M)=n \geq 3$, i.e.,
\begin{equation}\label{eq:a6} g(U,U)=g(V,V)=g(T,T)=1, g(U,V)=g(U,T)=g(V,T)=0.\end{equation}
\end{definition}
$a$, $b$, $c$ are called the associated scalars, $A$, $B$, $D$ are called the associated $1$-forms and $U$, $V$, $T$ are the generators of the manifold. Such an $n$-dimensional manifold is denoted by the symbol $E^n_{\rm{EQ}}$. If $c=0$, then become a quasi-Einstein manifolds. If $b=0$, then become a mixed quasi-Einstein manifolds. If $A=B$ or $A=D$, then become a generalized quasi-Einstein manifolds \textsuperscript{\cite{article.29}}. Let $E(X,Y)=B(X)D(Y)+D(X)B(Y)$, then $E$ is a symmetric $(0,2)$ tensor and satisfying $E(X,U)=0$ for all vector fields $X$. If $E$ is a symmetric $(0,2)$ tensor with trace free, then become a pseudo quasi-Einstein manifolds.

\subsection{Existence theorem}\

We state and prove the existence theorem of the extended quasi-Einstein manifold.
\begin{theorem}
If the Ricci tensor of a non-flat Riemannian manifold is non-vanishing and satisfies the relation
\begin{equation}\label{eq:b1} \begin{aligned}
\text{Ric}(X,W)g(Y,Z)&+\text{Ric}(Y,Z)g(X,W)={{a}_{1}}[g(X,W)g(Y,Z)+g(Y,W)g(X,Z)] \\
 & +{{a}_{2}}[\text{Ric}(X,Z)\text{Ric}^{2}(Y,W)+\text{Ric}(Y,W)\text{Ric}^{2}(X,Z)],
\end{aligned} \end{equation}
where ${{a}_{1}}$, ${{a}_{2}}$ are non-zero scalars. Then the manifold is an extended quasi-Einstein manifold.
\end{theorem}
\begin{proof}
Let $U$ be a non-null vector fields defined by $g(X,U)=\omega\, (X)$ for all vector fields $X$. Putting $X=W=U$ in \eqref{eq:b1}, we obtain
\begin{equation}\label{eq:b3} \begin{aligned}
\text{Ric}(Y,Z)g(U,U)&+\text{Ric}(U,U)g(Y,Z)={{a}_{1}}[g(U,U)g(Y,Z)+g(Y,U)g(Z,U)] \\
 & +{{a}_{2}}[\text{Ric}(U,Z)\text{Ric}^{2}(Y,U)+\text{Ric}(Y,U)\text{Ric}^{2}(X,U)].
\end{aligned} \end{equation}
From above equation \eqref{eq:b3},
\begin{equation}\label{eq:b4} \begin{aligned}
g(U,U)\text{Ric}(Y,Z)& =[{{a}_{1}}g(U,U)-\text{Ric}(U,U)]g(Y,Z)+{{a}_{1}}\omega (Y)\omega (Z)\\
& +{{a}_{2}}[\omega (QZ)\omega ({Q}^{2}Y)+\omega (QY)\omega ({Q}^{2}Z)],
\end{aligned} \end{equation}
where $\omega (Y)=g(Y,U)$, $\omega (QY)=\text{Ric}(Y,U)$ and  $\omega ({Q}^{2}Y)=\text{Ric}^{2}(Y,U)$. $Q$ is the symmetric endomorphism of the tangent space at each point corresponding to the Ricci tensor, i.e., $g(QX,Y)=\text{Ric}(X,Y)$, $\text{Ric}(QX,Y)=\text{Ric}^{2}(X,Y)$.
Since $U$ is non-null, we have $g(U,U)\ne 0$. There
\begin{equation}\label{eq:b5} \begin{aligned}
\text{Ric}(Y,Z)& =[{{a}_{1}}-\frac{\text{Ric}(U,U)}{g(U,U)}]g(Y,Z)+\frac{{a}_{1}}{g(U,U)}\omega (Y)\omega (Z)\\
& +\frac{{a}_{2}}{g(U,U)}[\omega (QZ)\omega ({Q}^{2}Y)+\omega (QY)\omega ({Q}^{2}Z)].
\end{aligned} \end{equation}
Taking $\omega (QY)={\omega _{1}}(Y)$ and $\omega ({Q}^{2}Y)={\omega _{2}}(Y)$, we get
\begin{equation}\label{eq:b6} \begin{aligned}
\text{Ric}(Y,Z)& =ag(Y,Z)+b\omega (Y)\omega (Z)+c[{{\omega }_{1}}(Y){{\omega }_{2}}(Z)+{{\omega }_{2}}(Y){{\omega }_{1}}(Z)],
\end{aligned} \end{equation}
where $a={{a}_{1}}-\frac{\text{Ric}(U,U)}{g(U,U)}$, $b=\frac{{a}_{1}}{g(U,U)}$, $c=\frac{{a}_{2}}{g(U,U)}$. This shows that the manifold is an extended quasi-Einstein manifold.
\end{proof}

\subsection{Relationship between the associated scalars}\

From \eqref{eq:a6}, we have
\begin{equation}\label{eq:b7} A(V)=A(T)=B(U)=B(T)=D(U)=D(V)=0. \end{equation}
Putting $Y=U$ in \eqref{eq:a4}, have
\begin{equation}\label{eq:b8} \text{Ric}(X,U)=(a+b)A(X). \end{equation}
Putting $Y=V$ in \eqref{eq:a4}, we get
\begin{equation}\label{eq:b9} \text{Ric}(X,V)=aB(X)+cD(X). \end{equation}
Again, putting $Y=T$ in \eqref{eq:a4}, we have
\begin{equation}\label{eq:b10} \text{Ric}(X,T)=aD(X)+cB(X). \end{equation}
Further,
\begin{equation}\label{eq:b11} \text{Ric}(U,U)=a+b,\quad \text{Ric}(V,V)=\text{Ric}(T,T)=a, \end{equation}
\begin{equation}\label{eq:b12} \text{Ric}(U,V)=\text{Ric}(U,T)=0,\quad \text{Ric}(T,V)=c. \end{equation}
Contracting $X$ and $Y$ in \eqref{eq:a4}, gives
\begin{equation}\label{eq:b13} r=na+b. \end{equation}

Next, let $Q$ be the symmetric endomorphism of the tangent space at a point corresponding to the Ricci tensor, then
\begin{equation}\label{eq:b14} g(QX,Y)=\text{Ric}(X,Y), \end{equation}
for any $X$, $Y$ $\in {{C}^{\infty }}(TM)$. Let ${{l}^{2}}$ be the square of the length of the Ricci tensor, then
\begin{equation}\label{eq:b15} {{l}^{2}}=\sum\limits_{i=1}^{n}{\text{Ric}(Q{{e}_{i}},{{e}_{i}})}, \end{equation}
where $\{{e}_{i}\}$, $i=1,2,3,\cdot \cdot \cdot ,n$ is an orthogonal basis of the tangent space at a point.
\begin{proposition}
In an extended quasi-Einstein manifold, if $a\ne 0$, then the associated scalar $c$ is less than $\frac{1}{\sqrt{2}}l$. \end{proposition}
\begin{proof}From \eqref{eq:a4}, we have
\begin{equation}\label{eq:b16}
\text{Ric}(Q{{e}_{i}},{{e}_{i}})=ag(Q{{e}_{i}},{{e}_{i}})+bA(Q{{e}_{i}})A({{e}_{i}})+c[B(Q{{e}_{i}})D({{e}_{i}})+D(Q{{e}_{i}})B({{e}_{i}})].
\end{equation}
Hence ${{l}^{2}}=(n-1){{a}^{2}}+{{(a+b)}^{2}}+2{{c}^{2}}.$ Since $a\ne 0$, then $2{{c}^{2}<{{l}^{2}}},$ i.e., $c<\frac{1}{\sqrt{2}}l.$
\end{proof}

\subsection{The generators as Killing vector fields}\

A vector field $X={{X}^{i}}\frac{\partial }{\partial {{x}^{i}}}$ is a Killing vector field if and only if it satisfies
\[\Delta {{X}^{i}}+R_{j}^{i}{{X}^{j}}=0, div(X)=0,\]
where $X$ is a fixed vector field on $M$.

For \eqref{eq:a4}, if the associated scalars $a$, $b$, $c$ are constants, then
\begin{equation}\label{eq:b17} \begin{aligned}
({{\nabla }_{X}}\text{Ric})(Y,Z)=&b[({{\nabla }_{X}}A)(Y)A(Z)+A(Y)({{\nabla }_{X}}A)(Z)]\\
&+c[({{\nabla }_{X}}B)(Y)D(Z)+B(Y)({{\nabla }_{X}}D)(Z)\\
&+({{\nabla }_{X}}D)(Y)B(Z)+D(Y)({{\nabla }_{X}}B)(Z)],\\
({{\nabla }_{Y}}\text{Ric})(Z,X)=&b[({{\nabla }_{Y}}A)(Z)A(X)+A(Z)({{\nabla }_{Y}}A)(X)]\\
&+c[({{\nabla }_{Y}}B)(Z)D(X)+B(Z)({{\nabla }_{Y}}D)(X) \\
&+({{\nabla }_{Y}}D)(Z)B(X)+D(Z)({{\nabla }_{Y}}B)(X)],\\
({{\nabla }_{Z}}\text{Ric})(X,Y)=& b[({{\nabla }_{Z}}A)(X)A(Y)+A(X)({{\nabla }_{Z}}A)(Y)]\\
&+c[({{\nabla }_{Z}}B)(X)D(Y)+B(X)({{\nabla }_{Z}}D)(Y) \\
&+({{\nabla }_{Z}}D)(X)B(Y)+D(X)({{\nabla }_{Z}}B)(Y)],
\end{aligned} \end{equation}
for any $X$, $Y$, $Z$ $\in {{C}^{\infty }}(TM).$

\begin{theorem}
If the generators $U$, $V$ and $T$ are Killing vector field, then the extended quasi-Einstein manifold satisfies cyclic parallel Ricci tensor. \end{theorem}
\begin{proof}
Let us suppose that the generator $U$ of the manifold is a Killing vector field, then we have \[({{\mathcal{L}}_{U}}g)(X,Y)=0,\]
where $\mathcal{L}$ denotes the Lie derivative. From which we get
\begin{equation}\label{eq:b18} g({{\nabla }_{X}}U,Y)+g(X,{{\nabla }_{Y}}U)=0. \end{equation}
Again since $g({{\nabla }_{X}}U,Y)=({{\nabla }_{X}}A)(Y)$, we get that
\begin{equation}\label{eq:b19} ({{\nabla }_{X}}A)(Y)+({{\nabla }_{Y}}A)(X)=0, \end{equation}
for any $X$, $Y$ $\in {{C}^{\infty }}(TM)$. Similarly, we have
\begin{equation}\label{eq:b20}
({{\nabla }_{X}}A)(Z)+({{\nabla }_{Z}}A)(X)=0,\quad ({{\nabla }_{Z}}A)(Y)+({{\nabla }_{Y}}A)(Z)=0. \end{equation}
Further, we suppose that the generator $V$, $T$ of the manifold are Killing vector field, too, then we have
\begin{equation}\label{eq:b21} \begin{aligned}
& ({{\nabla }_{X}}B)(Y)+({{\nabla }_{Y}}B)(X)=0,\quad ({{\nabla }_{X}}D)(Y)+({{\nabla }_{Y}}D)(X)=0, \\
& ({{\nabla }_{X}}B)(Z)+({{\nabla }_{Z}}B)(X)=0,\quad ({{\nabla }_{X}}D)(Z)+({{\nabla }_{Z}}D)(X)=0, \\
& ({{\nabla }_{Y}}B)(Z)+({{\nabla }_{Z}}B)(Y)=0,\quad ({{\nabla }_{Y}}D)(Z)+({{\nabla }_{Z}}D)(Y)=0.
\end{aligned} \end{equation}
Using \eqref{eq:b19}, \eqref{eq:b20}, \eqref{eq:b21} and \eqref{eq:b17}, we get
\begin{equation}\label{eq:b22}({{\nabla }_{X}}\text{Ric})(Y,Z)+({{\nabla }_{Y}}\text{Ric})(Z,X)+({{\nabla }_{Z}}\text{Ric})(X,Y)=0.\end{equation}
\end{proof}

\subsection{The generators as parallel vector fields}\

A Riemannian manifold ${{M}^{n}}$ is said to be Ricci-recurrent \textsuperscript{\cite{article.5}}, if the Ricci tensor \text{Ric} is non-zero and satisfying the condition
\begin{equation}\label{eq:b25} ({{\nabla }_{X}}\text{Ric})(Y,Z)=\alpha (X)\text{Ric}(Y,Z), \end{equation}
for any $X$, $Y$, $Z$ $\in {{C}^{\infty }}(TM)$, where $\alpha $ is a non-zero $1$-form.

\begin{theorem}
If the generators $U$, $V$, $T$ are parallel vector field and $B\ne D$, then the extended quasi-Einstein manifold is Ricci-recurrent manifold.
\end{theorem}
\begin{proof}
We consider $\nabla U=0$, $\nabla V=0$, $\nabla T=0$, then
\begin{equation}\label{eq:b26} \text{Ric}(X,U)=0,\quad \text{Ric}(X,V)=0,\quad \text{Ric}(X,T)=0.\end{equation}
From \eqref{eq:b8}, \eqref{eq:b9}, \eqref{eq:b10} and \eqref{eq:b26}, we have
\begin{equation}\label{eq:b27} (a+b)A(X)=0,\quad (a-c)(B(X)-D(X))=0.\end{equation}
$A$ is a non-zero $1$-form, then $a=-b$. Since $a=c$, the equation \eqref{eq:a4} can be expressed as
\begin{equation}\label{eq:b28}\text{Ric}(X,Y)=c[g(X,Y)-A(X)A(Y)+B(X)D(Y)+D(X)B(Y)].\end{equation}
So
\begin{equation}\label{eq:b29} \begin{aligned}
({{\nabla }_{Z}}\text{Ric})(X,Y)& =({{\nabla }_{Z}}c)[g(X,Y)-A(X)A(Y)+B(X)D(Y)+D(X)B(Y)]\\
 & -c[-({{\nabla }_{Z}}A)(X)A(Y)-A(X)({{\nabla }_{Z}}A)(Y)+({{\nabla }_{Z}}B)(X)D(Y) \\
 & +B(X)({{\nabla }_{Z}}D)(Y)+({{\nabla }_{Z}}D)(X)B(Y)+D(X)({{\nabla }_{Z}}B)(Y)].
\end{aligned}\end{equation}
Since $U$, $V$, $T$ are parallel vector field, we have
\[({{\nabla }_{Z}}A)X=0,\quad ({{\nabla }_{Z}}B)X=0,\quad ({{\nabla }_{Z}}D)X=0.\]
Putting the above formula in \eqref{eq:b29}, we get
\[({{\nabla }_{Z}}\text{Ric})(X,Y)=({{\nabla }_{Z}}c)[g(X,Y)-A(X)A(Y)+B(X)D(Y)+D(X)B(Y)].\]
Let $F(Z)=\frac{{{\nabla }_{Z}}c}{c}$, then \[({{\nabla }_{Z}}\text{Ric})(X,Y)=F(Z)\text{Ric}(X,Y).\]
\end{proof}

\subsection{The generators as concurrent vector fields}\

A vector field ${\xi }$ is said to be concurrent \cite{article.2} if ${{\nabla }_{X}}{\xi}=\alpha X,$ where $\alpha $ is a nonzero constant. If $\alpha =0$, the vector field reduces to a parallel vector field.
\begin{theorem}
If the associated vector fields of a $E^n_{\rm{EQ}}$ are concurrent vector fields and the associated scalars $a$, $b$ are constants, then the manifold reduces to a mixed generalized quasi-Einstein manifold \cite{article.27}.
\end{theorem}
\begin{proof}
We consider the vector fields $U$, $V$ and $T$ corresponding to the associated $1$-forms $A$, $B$ and $D$, respectively, are concurrent. Then
\begin{equation}\label{eq:b37} ({{\nabla }_{X}}A)(Y)=\alpha g(X,Y), \end{equation}
\begin{equation}\label{eq:b38} ({{\nabla }_{X}}B)(Y)=\beta g(X,Y), \end{equation}
and
\begin{equation}\label{eq:b39} ({{\nabla }_{X}}D)(Y)=\gamma g(X,Y), \end{equation}
where $\alpha $, $\beta $ and $\gamma $ are nonzero constants.

Using \eqref{eq:b37}, \eqref{eq:b38}, \eqref{eq:b39} to
\begin{align*}
({{\nabla }_{X}}\text{Ric})(Y,Z)=&b[({{\nabla }_{X}}A)(Y)A(Z)+A(Y)({{\nabla }_{X}}A)(Z)]\\
& +c[({{\nabla }_{X}}B)(Y)D(Z)+B(Y)({{\nabla }_{X}}D)(Z)\\
& +({{\nabla }_{X}}D)(Y)B(Z)+D(Y)({{\nabla }_{X}}B)(Z)],\\
\end{align*}
we get
\begin{equation}\label{eq:b40}\begin{aligned}
({{\nabla }_{X}}\text{Ric})(Y,Z)=&b[\alpha g(X,Y)A(Z)+\alpha g(X,Z)A(Y)]\\
& +c[\beta g(X,Y)D(Z)+\gamma g(X,Z)B(Y)\\
& +\gamma g(X,Y)B(Z)+\beta g(X,Z)D(Y)].\\
\end{aligned}\end{equation}
Contracting \eqref{eq:b40} over $Y$ and $Z$ , we obtain
\begin{equation}\label{eq:b41} dr(X)=2b\alpha A(X)+2c[\beta D(X)+\gamma B(X)], \end{equation}
where $r$ is the scalar curvature of the manifold. Since $a$, $b$ $\in \mathbb{R}$, we obtain $dr(X)=0$, for all $X$. Then
\[b\alpha A(X)+c[\beta D(X)+\gamma B(X)]=0.\]
Since $b$ and $\alpha $ are nonzero constants, we have
\begin{equation}\label{eq:b42} A(X)=-\frac{c\beta }{b\alpha }D(X)-\frac{c\gamma }{b\alpha }B(X). \end{equation}
Using \eqref{eq:b42} in \eqref{eq:a4}, we obtain
\begin{equation}\label{eq:b43} \text{Ric}(X,Y)=a_{1}g(X,Y)+a_{2}B(X)B(Y)+a_{3}D(X)D(Y)+a_{4}[B(X)D(Y)+D(X)B(Y)], \end{equation}
where $a_{1}=a$, $a_{2}=\frac{(c\gamma )^{2}}{(b\alpha )^{2}}$, $a_{3}=\frac{(c\beta )^{2}}{(b\alpha )^{2}}$ and $a_{4}=c+\frac{c^{2}\beta \gamma }{(b\alpha )^{2}}$. The manifold reduces to a mixed generalized quasi-Einstein manifold.
\end{proof}

\subsection{Codazzi type of Ricci tensor}\

A Riemannian manifold $({{M}^{n}},g)$ is said to satisfy Codazzi type of Ricci tensor \textsuperscript{\cite{article.18}}, if its Ricci tensor \text{Ric} satisfies the following condition
\begin{equation}\label{eq:b30}
({{\nabla }_{X}}\text{Ric})(Y,Z)=({{\nabla }_{Y}}\text{Ric})(X,Z), \end{equation}
for any $X$, $Y$ $\in {{C}^{\infty }}(TM)$.

\begin{proposition}
If an extended quasi-Einstein manifold satisfies the Codazzi type of Ricci tensor and the generator $U$ is a concurrent vector field, then the associated $1$-form $A$ is closed.
\end{proposition}
\begin{proof}
An extended quasi-Einstein manifold satisfies the Codazzi type of Ricci tensor, then the Ricci tensor satisfies \eqref{eq:b30}.
Putting $Z=U$ in \eqref{eq:b30}, we have
\begin{align*}
& b[({{\nabla }_{X}}A)(Y)]+c[B(Y)({{\nabla }_{X}}D)(U)+D(Y)({{\nabla }_{X}}B)(U)]\\
&=b[({{\nabla }_{Y}}A)(X)]+c[B(X)({{\nabla }_{Y}}D)(U)+D(X)({{\nabla }_{Y}}B)(U)],\end{align*}
where the associated scalars $a$, $b$, $c$ are constants. Therefore
\begin{equation}\label{eq:b31}\begin{aligned}
 & b[({{\nabla }_{X}}A)(Y)-({{\nabla }_{Y}}A)(X)]\\
 & +c[B(X)D({{\nabla }_{Y}}U)+D(X)B({{\nabla }_{Y}}U)-B(Y)D({{\nabla }_{X}}U)-D(Y)B({{\nabla }_{X}}U)]=0.
\end{aligned}\end{equation}
Next, let the generator $U$ is a concurrent vector field \textsuperscript{\cite{article.2}}, then
\begin{equation}\label{eq:b32}
{{\nabla }_{X}}U=\alpha X, \end{equation}
where $\alpha\ $ is a non-zero constant.
Using \eqref{eq:b32} in \eqref{eq:b31}, we have
\begin{align*}
&b[({{\nabla }_{X}}A)(Y)-({{\nabla }_{Y}}A)(X)]\\
&+\alpha c[B(X)D(Y)+D(X)B(Y)-B(Y)D(X)-D(Y)B(X)]=0.\end{align*}
So \[b[({{\nabla }_{X}}A)(Y)-({{\nabla }_{Y}}A)(X)]=0.\]
Since $b\ne 0$, we have $({{\nabla }_{X}}A)(Y)-({{\nabla }_{Y}}A)(X)=0,$ i.e., $dA(X,Y)=0.$
\end{proof}

\subsection{Ricci-semi symmetric manifold}\

A Riemannian manifold $({{M}^{n}},g)$ is said to be Ricci-semi symmetric \textsuperscript{\cite{article.6}}, if the riemannian tensor and Ricci tensor satisfying the condition
\begin{equation}\label{eq:b33} R\cdot \text{Ric}=0.\end{equation}
The condition $R\cdot \text{Ric}$ can express as
\begin{equation}\label{eq:b34}\begin{aligned}
 (R(X,Y)\cdot \text{Ric})(Z,W)=&-\text{Ric}(R(X,Y)Z,W)-\text{Ric}(Z,R(X,Y)W) \\
 =&-b[A(R(X,Y)Z)A(W)+A(Z)A(R(X,Y)W)]\\
 &-c[B(R(X,Y)Z)D(W)+D(R(X,Y)Z)B(W) \\
 & +B(Z)D(R(X,Y)W)+D(Z)B(R(X,Y)W)],
\end{aligned}\end{equation}
for any $X$, $Y$, $Z$, $W$ $\in {{C}^{\infty }}(TM)$.
\begin{proposition}
If an extended quasi-Einstein manifold is a Ricci-semi symmetric manifold, then $R(X,Y,V,T)=0$. \end{proposition}
\begin{proof}
Now we suppose that an extended quasi-Einstein manifold is Ricci-semi symmetric manifold, i.e., such manifold satisfies \eqref{eq:b33}. Then
\begin{equation}\label{eq:b35}\begin{aligned}
& b[A(R(X,Y)Z)A(W)+A(Z)A(R(X,Y)W)]+c[B(R(X,Y)Z)D(W) \\
& +D(R(X,Y)Z)B(W)+B(Z)D(R(X,Y)W)+D(Z)B(R(X,Y)W)]=0.
\end{aligned}\end{equation}
Putting $W=Z=V$ in \eqref{eq:b35}, then we have \[2c[ D(R(X,Y)V)]=0.\]
Putting $W=Z=T$ in \eqref{eq:b35}, yields \[2c[ B(R(X,Y)T)]=0.\]
Since $c\ne 0$,
\begin{equation}\label{eq:b36}D(R(X,Y)V)=0, B(R(X,Y)T)=0.\end{equation}
Due to \eqref{eq:b36}, then $R(X,Y,V,T)=0.$
\end{proof}

\section{Two solitons structures}

In 2014, Nurowski and Randall introduced the concept of the generalised Ricci soliton equations \textsuperscript{\cite{article.14}}. These equations depend on three real parameters. Siddiqi MD study generalised Ricci solitons on trans sasakian manifolds \textsuperscript{\cite{article.13}}. As a generalied of Ricci soliton, Hiricu and Udriste introduced and studied Riemann soliton \textsuperscript{\cite{article.16}}. Riemann solitons are generalized fixed points of the Riemann flow. In 2020, Venkatesha study Riemann solitons and almost Rieman solitons on almost Kenmotsu manifolds \textsuperscript{\cite{article.15}}.

A generalised Ricci soliton is a (pseudo)-Riemannian manifold $({{M}^{n}},g)$ admitting a smooth vector field $X$, such that
\begin{equation}\label{eq:c1} {{\mathcal{L}}_{X}}g+2{{c}_{1}}{{X}^{b}}\odot {{X}^{b}}=2{{c}_{2}}\text{Ric}+2\lambda g, \end{equation}
for arbitrary real constant ${c}_{1}$, ${c}_{2}$ and $\lambda$. Here ${{\mathcal{L}}_{X}}g$ is the Lie derivative of the metric $g$ with respect to $X$, ${{X}^{b}}$ is a non-zero $1$-form such that ${{X}^{b}}(Y) =g(X,Y)$, \text{Ric} is the Ricci tensor of $g$. We call \eqref{eq:c1} the generalised Ricci soliton equations. A pair $(g,X)$ is called a generalised Ricci soliton if \eqref{eq:c1} is satisfied.

A Riemann soliton are defined by a smooth vector field $V$ and a real constant $\lambda$ which satisfy the following equation
\begin{equation}\label{eq:c2} R+\frac{1}{2}{{\mathcal{L}}_{V}}g\wedge g=\frac{\lambda }{2}g\wedge g,\end{equation}
where ${{\mathcal{L}}_{V}}g$ denotes the Lie derivative of $g$ and $\wedge$ is the Kulkarni-Nomizu product.

A vector field $\varphi $ on a Riemannian manifold $({{M}^{n}},g)$ is said to be a $\varphi (\text{Ric})$-vector field \textsuperscript{\cite{article.19}} if it satisfies
\begin{equation}\label{eq:c3} {{\nabla }_{X}}\varphi =\mu QX,\end{equation}
where $\mu $ is a constant and $Q$ is the ricci operator defined by $\text{Ric}(X,Y)=g(QX,Y)$. If $\mu \ne 0$, then the vector field $\varphi $ is called a proper $\varphi (\text{Ric})$-vector field.

Next, we consider an extended quasi-Einstein manifold with a generalised Ricci soliton and Riemannian soliton, respectively.

\subsection{Generalised Ricci soliton}
\begin{proposition}
Let $M$ be an extended quasi-Einstein manifold with a generalised Ricci soliton $(g,U)$ such that the vector field $U$ is the generator of $M$. Then the integral curves of $U$ are geodesic on $M$.
\end{proposition}
\begin{proof}
Since $(g,U)$ is a generalised Ricci soliton on $M$, from \eqref{eq:c1} we have
\begin{equation}\label{eq:c4} ({{\mathcal{L}}_{U}}g)(Y,Z)+2{{c}_{1}}g(U,Y)g(U,Z)=2{{c}_{2}}\text{Ric}(Y,Z)+2\lambda g(Y,Z),\end{equation}
for any $Y$, $Z$ $\in {{C}^{\infty }}(TM)$. Putting $Y=U$ in \eqref{eq:c4} and from \eqref{eq:b8} gives
\begin{equation}\label{eq:c5} g({{\nabla }_{U}}U,Z)=[2{{c}_{2}}(a+b)+2\lambda -2{{c}_{1}}]A(Z). \end{equation}
Putting $Z=U$ in \eqref{eq:c5}, we have
\begin{equation}\label{eq:c6} \lambda ={{c}_{1}}-{{c}_{2}}(a+b). \end{equation}
Therefore, from \eqref{eq:c5} and \eqref{eq:c6} we obtain
\begin{equation}\label{eq:c7} g({{\nabla }_{U}}U,Z)=0, \end{equation}
which implies ${{\nabla }_{U}}U=0$, the integral curves of the vector field $U$ are geodesic.
\end{proof}

\begin{proposition}
Let $M$ be an extended quasi-Einstein manifold with a generalised Ricci soliton $(g,V)$ such that the vector field $V$ is the generator of $M$. If ${{c}_{2}}\ne 0$, then the integral curves of $V$ are geodesic on $M$ if and only if the manifold $M$ is a quasi-Einstein manifold.
\end{proposition}
\begin{proof}
Since $(g,V)$ is a generalised Ricci soliton on $M$, from \eqref{eq:c1} we have
\begin{equation}\label{eq:c8} ({{\mathcal{L}}_{V}}g)(Y,Z)+2{{c}_{1}}g(V,Y)g(V,Z)=2{{c}_{2}}\text{Ric}(Y,Z)+2\lambda g(Y,Z),\end{equation}
for any $Y$, $Z$ $\in {{C}^{\infty }}(TM)$. Putting $Y=V$ in \eqref{eq:c8} and from \eqref{eq:b9} yields
\begin{equation}\label{eq:c9} g({{\nabla }_{V}}V,Z)=2[a{{c}_{2}}+\lambda -{{c}_{1}}]B(Z)+2c{{c}_{2}}D(Z). \end{equation}
Putting $Z=V$ in \eqref{eq:c9}, we obtain
\begin{equation}\label{eq:c10} \lambda ={{c}_{1}}-a{{c}_{2}}. \end{equation}
From \eqref{eq:c9} and \eqref{eq:c10}, we obtain
\begin{equation}\label{eq:c11} g({{\nabla }_{V}}V,Z)=2c{{c}_{2}}D(Z).\end{equation}

Let us suppose that the integral curves of $V$ are geodesic on $M$. Then, \eqref{eq:c11} becomes
\begin{equation}\label{eq:c12} 2c{{c}_{2}}D(Z)=0. \end{equation}
Taking $Z=T$ in \eqref{eq:c12} and ${{c}_{2}}\ne 0$, we find $c=0$. This implies that $M$ is a quasi-Einstein manifolds.

Conversely, we assume that $M$ is a quasi-Einstein manifolds. Then, we have $c=0$. Hence, we get \[g({{\nabla }_{V}}V,Z)=0.\]
which implies that ${{\nabla }_{V}}V=0$.
\end{proof}

\begin{proposition}
An extended quasi-Einstein manifold $M$ admitting a generalised Ricci soliton $(g,V)$, then $divV=n(2r-a){{c}_{2}}$.
\end{proposition}
\begin{proof}
By the suitable contraction of \eqref{eq:c8}, we obtain
\begin{equation}\label{eq:c13} \lambda =\frac{1}{n}divV+{{c}_{1}}-2{{c}_{2}}. \end{equation}
Using \eqref{eq:c10} and \eqref{eq:c13}, we get \[divV=n(2r-a){{c}_{2}},\]
where $r$ is scalar curvature.
\end{proof}

\begin{theorem}
Let $M$ be an extended quasi-Einstein manifold with a generalised Ricci soliton $(g,V)$ and the vector field $V$ is a $V$(\text{Ric})-vector field. Then, $M$ is either a quasi-Einstein manifold or the generalised Ricci soliton reduces to a steady Ricci soliton.
\end{theorem}
\begin{proof}
We consider $V$ and $T$ are mutually orthogonal unit vector fields, i.e., $t=0$. It follows from the Lie-derivative and from \eqref{eq:c3}, we have
\begin{equation}\label{eq:c14} ({{\mathcal{L}}_{V}}g)(X,Y)=2\mu \text{Ric}(X,Y), \end{equation}
for any $X$, $Y$ $\in {{C}^{\infty }}(TM)$. $M$ is an extended quasi-Einstein manifold with generalised Ricci soliton $(g,V)$. From \eqref{eq:c1} and \eqref{eq:c14} we find
\begin{equation}\label{eq:c15} (\mu -{{c}_{2}})\text{Ric}(X,Y)=\lambda g(X,Y)-{{c}_{1}}B(X)B(Y). \end{equation}
Taking $X=Y=U$ in \eqref{eq:c15} gives
\begin{equation}\label{eq:c16} (a+b)(\mu -{{c}_{2}})=\lambda. \end{equation}
Taking $X=Y=V$ in \eqref{eq:c15} gives
\begin{equation}\label{eq:c17} a(\mu -{{c}_{2}})=\lambda -{{c}_{1}}. \end{equation}
Taking $X=V$, $Y=T$ in \eqref{eq:c15} gives
\begin{equation}\label{eq:c18} c(\mu -{{c}_{2}})=0. \end{equation}
which implies that $\mu ={{c}_{2}}$ or $c=0$.

If $c=0$, then $M$ is a quasi-Einstein manifold. If $\mu ={{c}_{2}}$, then from \eqref{eq:c16} and \eqref{eq:c17} we have
$\lambda ={{c}_{1}}=0$. This means that the generalised Ricci soliton reduces to a steady Ricci soliton.
\end{proof}

\subsection{Riemann soliton}
\begin{proposition}
Let $M$ be an extended quasi-Einstein manifold with a Riemann soliton $(g,V)$. Then, the integral curves of $V$ are geodesic on $M$ if and only if the manifold M is a quasi-Einstein manifold.
\end{proposition}
\begin{proof}
The Riemannian soliton equation \eqref{eq:c2} can be expressed as
\begin{equation}\label{eq:c19}\begin{aligned}
  2R(X,Y,Z,W)& +\{g(X,W)({{\mathcal{L}}_{V}}g)(Y,Z)+g(Y,Z)({{\mathcal{L}}_{V}}g)(X,W) \\
 & -g(X,Z)({{\mathcal{L}}_{V}}g)(Y,W)-g(Y,W)({{\mathcal{L}}_{V}}g)(X,Z)\} \\
 & =2\lambda \{g(X,W)g(Y,Z)-g(X,Z)g(Y,W)\}.
\end{aligned}\end{equation}
Contracting \eqref{eq:c19} over $X$ and $W$, we obtain
\begin{equation}\label{eq:c20}
({{\mathcal{L}}_{V}}g)(Y,Z)+\frac{2}{n-2}\text{Ric}(Y,Z)-\frac{2}{n-2}[(n-1)\lambda -\text{div}V]g(Y,Z)=0.
\end{equation}
Putting $Y=V$ in \eqref{eq:c20}, we obtain
\begin{equation}\label{eq:c21}\begin{aligned}
  g({{\nabla }_{V}}V,Z)+\frac{2}{n-2}[aB(Z)+cD(Z)]& \\
 -\frac{2}{n-2}[(n-1)\lambda -\text{div}V]B(Z)& =0.
\end{aligned}\end{equation}
Putting $Z=V$ in \eqref{eq:c21}, we have $a=(n-1)\lambda -\text{div}V$. Then, \eqref{eq:c21} become
\begin{equation}\label{eq:c22} g({{\nabla }_{V}}V,Z)+\frac{2}{n-2}cD(Z)=0. \end{equation}
Now, let us suppose that the integral curves of $V$ are geodesic on $M$. Then, \eqref{eq:c22} becomes
\begin{equation}\label{eq:c23} \frac{2}{n-2}cD(Z)=0.\end{equation}
Taking $Z=T$ in \eqref{eq:c23}, we find $c=0$. This implies that $M$ is a quasi-Einstein mainfold.

Conversely, we assume that $M$ is a quasi-Einstein manifold. Then, we have $c=0$. Hence, we get \[g({{\nabla }_{V}}V,Z)=0,\]
which implies that ${{\nabla }_{V}}V=0$.
\end{proof}

\begin{proposition}
Let $M$ be an extended quasi-Einstein manifold with a Riemann soliton $(g,U)$ such that the vector field $U$ is the generator of $M$. Then the integral curves of $U$ are geodesic on $M$.
\end{proposition}
\begin{proof}
From \eqref{eq:c20}, we have
\begin{equation}\label{eq:c24}
({{\mathcal{L}}_{U}}g)(Y,Z)+\frac{2}{n-2}\text{Ric}(Y,Z)-\frac{2}{n-2}[(n-1)\lambda -\text{div}U]g(Y,Z)=0.
\end{equation}
Substituting $Y=U$ into \eqref{eq:c24}, yields
\begin{equation}\label{eq:c25}g({{\nabla }_{U}}U,Z)+\frac{2}{n-2}[(a+b)A(Z)]-\frac{2}{n-2}[(n-1)\lambda -\text{div}U]A(Z)=0.\end{equation}
Putting $Z=U$ in \eqref{eq:c25}, we have $a+b=(n-1)\lambda -\text{div}U$. Then, \eqref{eq:c25} become
\begin{equation}\label{eq:c26} g({{\nabla }_{U}}U,Z)=0, \end{equation}
which implies ${{\nabla }_{U}}U=0$, the integral curves of the vector field $U$ are geodesic.
\end{proof}

\begin{proposition}
Let $M$ be an extended quasi-Einstein manifold admitting a Riemann soliton $(g,V)$ such that the vector field $V$ is a $V$(\text{Ric})-vector field. Then $V$ is a proper $V$(\text{Ric})-vector field, and $\mu =-\frac{1}{1-2n}$.
\end{proposition}
\begin{proof}
Taking covariant derivative of the equation \eqref{eq:c20} gives
\begin{equation}\label{eq:c27} ({{\nabla }_{X}}{{\mathcal{L}}_{V}}g)(Y,Z)+\frac{2}{2n-1}({{\nabla }_{X}}\text{Ric})(Y,Z)=0, \end{equation}
where $V$ has a constant divergence.
Similarly, taking covariant derivative of the equation \eqref{eq:c14}, have
\begin{equation}\label{eq:c28} ({{\nabla }_{X}}{{\mathcal{L}}_{V}}g)(Y,Z)=2\mu ({{\nabla }_{X}}\text{Ric})(Y,Z).\end{equation}
From \eqref{eq:c26} and \eqref{eq:c28}, we have $\mu =-\frac{1}{1-2n}$, which implies $V$ is a proper $V$(\text{Ric})-vector field.
\end{proof}

\section{Example}

We consider a Riemannian manifold $({{M}^{4}},g)$ endowed with the Riemannian metric $g$ defined by
\begin{equation}\label{eq:e4}
d{{s}^{2}}={{g}_{ij}}d{{x}^{i}}d{{x}^{j}}={{x}^{2}}{{(d{{x}^{1}})}^{2}}
+{{x}^{1}}{{(d{{x}^{2}})}^{2}}+{{(d{{x}^{3}})}^{2}}+{{(d{{x}^{4}})}^{2}},
\end{equation}
where $i,j=1,2,3,4$. The non-vanishing components of Christoffel symbols are
\[\Gamma _{12}^{1}=\frac{1}{2{{x}^{2}}},\quad \Gamma _{22}^{1}=-\frac{1}{2{{x}^{2}}},\quad \Gamma _{11}^{2}=-\frac{1}{2{{x}^{1}}},\quad\Gamma _{12}^{2}=\frac{1}{2{{x}^{1}}}.\]
The non-zero curvature tensor are \[{{R}_{1212}}=\frac{1}{4{{x}^{1}}}-\frac{1}{4{{x}^{2}}}.\]
The non-vanishing components of Ricci tensor are
\[{{\text{Ric}}_{11}}=\frac{1}{4{{x}^{1}}{{x}^{2}}}-\frac{1}{4{({{x}^{1}})^{2}}},\quad {{\text{Ric}}_{22}}=\frac{1}{4{({{x}^{2}})^{2}}}-\frac{1}{4{{x}^{1}}{{x}^{2}}},\]
and the scalar curvature is \[r=\frac{1}{2{{x}^{1}}{({{x}^{2}})^{2}}}-\frac{1}{2{{x}^{2}}{({{x}^{1}})^{2}}}.\]
which is non zero and non-constant. We shall now show that this ${{M}^{4}}$ is an extended quasi-Einstein manifold, i.e., it satisfies the defining relation \eqref{eq:a4}.

Now, we take the associated scalars as follows: \[a=\frac{1}{{{x}^{1}}{{x}^{2}}},\quad b=-\frac{2}{{{x}^{1}}{{x}^{2}}},\quad c=-\frac{\sqrt{10}}{5{{x}^{1}}{{x}^{2}}[\frac{1}{{{x}^{2}}}-\frac{1}{{{x}^{1}}}]}.\]
We take the 1-forms as follows:
\[{{A}_{i}}(x)=\left\{ \begin{aligned}
  & \frac{\sqrt{2{{x}^{2}}}}{2},\quad for\,i=1 \\
 & \frac{\sqrt{2{{x}^{1}}}}{2},\quad for\,i=2 \\
 & 0,\quad\quad\quad for\,i=3,4
\end{aligned} \right.
,\quad {{B}_{i}}(x)=\left\{ \begin{aligned}
 & \frac{\sqrt{2{{x}^{2}}}}{4},\quad\quad\, for\,i=1 \\
 & -\frac{\sqrt{2{{x}^{1}}}}{4},\quad for\,i=2\\
 & 0,\quad\quad\quad\quad\,  for\,i=3 \\
 & \frac{\sqrt{3}}{2},\quad\quad \quad \, for\,i=4 \\
\end{aligned} \right.\]
and
\[{{D}_{i}}(x)=\left\{ \begin{aligned}
 & -\frac{\sqrt{5{{x}^{2}}}}{4},\quad for\,i=1 \\
 & \frac{\sqrt{5{{x}^{1}}}}{4},\quad\quad  for\,i=2\\
 & \frac{\sqrt{6}}{6},\quad\quad\quad for\,i=3 \\
 & \frac{\sqrt{30}}{12},\,\quad\quad\,\, for\,i=4
\end{aligned}, \right.\]
at any point $x\in M$. In $({{M}^{4}},g)$, \eqref{eq:a4} reduces with these associated scalars and $1$-forms to the following equation:
\begin{equation}\label{eq:e5} {\text{Ric}_{11}}=a{{g}_{11}}+b{{A}_{1}}{{A}_{1}}+c[{{B}_{1}}{{D}_{1}}+{{D}_{1}}{{B}_{1}}],\end{equation}
\begin{equation}\label{eq:e6} {\text{Ric}_{22}}=a{{g}_{22}}+b{{A}_{2}}{{A}_{2}}+c[{{B}_{2}}{{D}_{2}}+{{D}_{2}}{{B}_{2}}].\end{equation}
It can prove that \eqref{eq:e5}, \eqref{eq:e6} is true. We shall now show that the 1-forms are unit,
\[{{g}^{ij}}{{A}_{i}}{{A}_{j}}=1, {{g}^{ij}}{{B}_{i}}{{B}_{j}}=1, {{g}^{ij}}{{D}_{i}}{{D}_{j}}=1, {{g}^{ij}}{{A}_{i}}{{B}_{j}}=0,
 {{g}^{ij}}{{A}_{i}}{{D}_{j}}=0, {{g}^{ij}}{{B}_{i}}{{D}_{j}}=0.\]
So manifold is an extended quasi-Einstein manifold.

\section{Discussions and further questions}

In this final section, we give some related questions which deserve to be studied further.

\begin{enumerate}
  \item The concepts of quasi conformally curvature tensor, conformally curvature tensor, projective curvature tensor and conharmonic curvature tensor are introduced in \cite{article.21,article.22,article.23,article.24}. In the next step of our research, we will consider an extended quasi-Einstein manifolds with some flat condition. Further, we study some symmetric condition on the extended quasi-Einstein manifold, such as, semi-symmetric, conformally semi-symmetric, projectively semi-symmetric, concircularly semi-symmetric, conharmonically semi-symmetric \textsuperscript{\cite{article.7}}.
  \item We study Deszcz pseudo symmetric conditions on the extended quasi-Einstein manifold base on the study of equivalence of geometric structures in \cite{article.25} by considering different conditions into various groups or classes.
  \item Imposing an additional circulant structure on four-dimensional Riemannian manifolds and the components of the circulant structure in local coordinates are circulant matrices, such additional structures have been widely studied \textsuperscript{\cite{article.26}}. We will impose a skew-circulant structure on the extended quasi-Einstein manifolds, and study geometric properties of such manifold.
\end{enumerate}
\noindent {\bf{Acknowledgement.}} W. J. Lu gratefully acknowledges support by the Natural Science Foundation of China($\#$. 12061014) and the National Science Foundation of Guangxi Province($\#$. 2019GXNSFAA245043)

\clearpage
\end{CJK*}
\end{document}